\documentclass[12pt]{amsart}
\usepackage{amsmath,amsthm,amsfonts,amssymb,mathrsfs}
\usepackage{color}

\usepackage{tikz}

\usepackage{amssymb,cite}
\usepackage[colorlinks,plainpages,citecolor=magenta, linkcolor=blue, backref]{hyperref}

\usepackage{hyperref}
\date{\today}

\usepackage{hyperref}
\input{xy}
 \xyoption{all}
 \xyoption{arc}

\newtheorem{theorem}{Theorem}[section]

\newtheorem{proposition}[theorem]{Proposition}
\newtheorem{corollary}[theorem]{Corollary}

\newtheorem{lemma}[theorem]{Lemma}
\theoremstyle{definition}
\newtheorem{definition}[theorem]{Definition}
\newtheorem{construction}[theorem]{Construction}
\newtheorem{example}[theorem]{Example}
\newtheorem{remark}[theorem]{Remark}

\begin{document}

\title[On semitopological simple inverse $\omega$-semigroups with compact maximal subgroups]{On semitopological simple inverse $\omega$-semigroups with compact maximal subgroups}  
\author{Oleg Gutik and Kateryna Maksymyk}
\address{Ivan Franko National University of Lviv, Universytetska 1, Lviv, 79000, Ukraine}
\email{oleg.gutik@lnu.edu.ua, kateryna.maksymyk@lnu.edu.ua}



\begin{abstract}
We describe the structure of ($0$-)simple inverse Hausdorff semitopological $\omega$-semigroups with compact maximal subgroups. In particular, we show that if $S$ is a simple inverse Hausdorff semitopological $\omega$-semigroup with compact maximal subgroups, then $S$ is topologically isomorphic to the Bruck--Reilly extension $\left(\textbf{BR}(T,\theta),\tau_{\textbf{BR}}^{\oplus}\right)$ of a finite semilattice $T=\left[E;G_\alpha,\varphi_{\alpha,\beta}\right]$ of compact groups $G_\alpha$ in the class of topological inverse semigroups, where $\tau_{\textbf{BR}}^{\oplus}$ is the sum direct topology on $\textbf{BR}(T,\theta)$. Also we prove that every Hausdorff locally compact shift-continuous topology on the simple inverse Hausdorff semitopological $\omega$-semigroups with compact maximal subgroups with adjoined zero is either compact or the zero is an isolated point.

\end{abstract}

\maketitle


\section{Introduction, motivation and main definitions}

We shall follow the terminology of~\cite{Carruth-Hildebrant-Koch=1983, Clifford-Preston-1961, Clifford-Preston-1967, Engelking=1989, Lawson=1998, Ruppert=1984}. By $\omega$ we denote the set of all non-negative integers, by $\mathbb{N}$ the set of all positive integers. All topological spaces, considered in this paper, are Hausdorff, if the otherwise is not stated explicitly.
If $A$ is a subset of a topological space $X$ then by $\mathrm{cl}_X(A)$ and $\mathrm{int}_X(A)$ we denote the closure and interior of $A$ in $X$, respectively.

Let $\mathfrak{h}\colon S\to T$ be a map of sets. Then for any $s\in S$ and $A\subseteq S$ by $(s)\mathfrak{h}$ and $(A)\mathfrak{h}$ we denote the images of $s$ and $A$, respectively, under the map $\mathfrak{h}$. Also, for any $t\in T$ and $B\subseteq T$ by $(s)\mathfrak{h}^{-1}$ and $(B)\mathfrak{h}^{-1}$ we denote the full preimages of $t$ and $B$, respectively, under the map $\mathfrak{h}$.

A semigroup $S$ is called {\it inverse} if for any
element $x\in S$ there exists a unique $x^{-1}\in S$ such that
$xx^{-1}x=x$ and $x^{-1}xx^{-1}=x^{-1}$. The element $x^{-1}$ is
called the {\it inverse of} $x\in S$. If $S$ is an inverse
semigroup, then the function $\operatorname{inv}\colon S\to S$
which assigns to every element $x$ of $S$ its inverse element
$x^{-1}$ is called the {\it inversion}.

If $S$ is a semigroup, then we shall denote the subset of all
idempotents in $S$ by $E(S)$. If $S$ is an inverse semigroup, then
$E(S)$ is closed under multiplication and we shall refer to $E(S)$ as a
\emph{band} (or the \emph{band of} $S$). Then the semigroup
operation on $S$ determines the following partial order $\preccurlyeq$
on $E(S)$: $e\preccurlyeq f$ if and only if $ef=fe=e$. This order is
called the {\em natural partial order} on $E(S)$. A \emph{semilattice} is a commutative semigroup of idempotents. A \emph{chain} is a linearly ordered semilattice.

A semigroup $S$ is said to be \emph{simple} (\emph{$0$-simple}) if $S$ has no proper two-sided ideals (if $S$ has the zero $\mathbf{0}$ and $\{\mathbf{0}\}$ is the unique proper two-sided ideal of $S$). A semigroup $S$ is called an \emph{$\omega$-semigroup} if the band $E(S)$ is order isomorphic to $(\omega,\geqslant)$. Also, an inverse semigroup $S$ is $0$-simple $\omega$-semigroup if $S$ is $0$-simple and the subset of non-zero idempotents $E(S)\setminus\{\mathbf{0}\}$ is order isomorphic to $(\omega,\geqslant)$.


If $S$ is an inverse semigroup, then the semigroup operation on $S$ determines the following partial order $\preccurlyeq$
on $S$: $s\preccurlyeq t$ if and only if there exists $e\in E(S)$ such that $s=te$. This order is
called the {\em natural partial order} on $S$ \cite{Wagner-1952}.

The bicyclic monoid ${\mathscr{C}}(p,q)$ is the semigroup with the identity $1$ generated by two elements $p$ and $q$ subjected only to the condition $pq=1$. The semigroup operation on ${\mathscr{C}}(p,q)$ is determined as
follows:
\begin{equation*}
    q^kp^l\cdot q^mp^n=q^{k+m-\min\{l,m\}}p^{l+n-\min\{l,m\}}.
\end{equation*}
It is well known that the bicyclic monoid ${\mathscr{C}}(p,q)$ is a bisimple (and hence, simple) combinatorial $E$-unitary inverse semigroup and every non-trivial congruence on ${\mathscr{C}}(p,q)$ is a group congruence \cite{Clifford-Preston-1961}.

Using the construction of the bicyclic monoid Bruck build the construction of isomorphic embedding of any (inverse) semigroup into a simple inverse monoid (see \cite{Bruck-1958}   and \cite[Section 8.5]{Clifford-Preston-1967}). Later Reilly \cite{Reilly=1966} and  Warne \cite{Warne=1966} generalized the Bruck construction for the description of of the structure of bisimple regular $\omega$-semigroups in the following way.

\begin{construction}[{\cite{Reilly=1966,Warne=1966}}]
Let $S$ be monoid with the unit element $1_S$ and let $\theta\colon S\to H_{1_s}$ be a homomorphism from $S$ into the group  of units $H({1_s})$ of $S$. On the set $\mathbf{BR}(S,\theta)=\omega\times S\times \omega$ we define the semigroup operation by the formula
\begin{equation*}
  (i,s,j)\cdot (k.t,l)=\left(i+k-\min\{j,k\},(s)\theta^{k-\min\{j,k\}}(t)\theta^{j-\min\{j,k\}},j+l-\min\{j,k\}\right),
\end{equation*}
where $i,j,k,l\in\omega$, $s,t\in S$ and $\theta^0$ is the identity map on $S$. Then $\mathbf{BR}(S,\theta)$ with such defined semigroup operation is called the \emph{Bruck-Reilly extension} of $S$.
\end{construction}

Later we assume that $S$ is a monoid.

For arbitrary $i,j\in \omega$ and non-empty subset $A$ of the semigroup $S$ we define the subset $A_{i,j}$  of $\textbf{BR}(S,\theta)$ as follows: $A_{i,j}=\left\{(i, s, j)\colon s\in A\right\}$.

We observe that if S is a trivial monoid then $\mathbf{BR}(S,\theta)$ is isomorphic to the bicyclic semigroup
${\mathscr{C}}(p,q)$ and in case when $\theta$ is an annihilating homomorphism (i.e., $(s)\theta=1_S$), then $\mathbf{BR}(S)=\mathbf{BR}(S,\theta)$  is called the \emph{Bruck semigroup over monoid} $S$ \cite{Gutik=1994}. Also Reilly and Warne proved that every bisimple regular $\omega$-semigroup is isomorphic to the Bruck–Reilly extension of a some group \cite{Reilly=1966,Warne=1966}.

Later we need the following proposition  which is a simple generalization of Lemma~1.2 from \cite{Munn-Reilly=1966} and follows from Theorem~XI.1.1 of \cite{Petrich=1984}.

\begin{proposition}\label{proposition-1.3}
Let $S$ be an arbitrary monoid and  $\theta\colon S\to H_S(1)$ a homomorphism from $S$ into the group of units $H_S(1)$ of $S$. Then a map \emph{$\eta\colon \textbf{BR}(S,\theta)\to \mathscr{C}(p,q)$} defined by the formula $\eta(i, s, j)=q^ip^j$ is a homomorphism and hence the relation $\eta^\natural$ on \emph{$\textbf{BR}(S,\theta)$} defined in the following way
\begin{equation*}
  (i, s, j)\eta^\natural(m,t,n) \qquad \hbox{if and only if} \qquad i=m \quad \hbox{and} \quad j=n
\end{equation*}
is a congruence.
\end{proposition}

Later we need the following well-known construction.

\begin{construction}[\cite{Petrich=1984}]
Let $E$ be a semilattice. To each $\alpha\in E$ associate a semigroup $S_\alpha$ and assume that $S_\alpha\cap S_\beta=\varnothing$ if $\alpha\neq\beta$. For each pair  $\beta\preccurlyeq\alpha$, let $\varphi_{\alpha,\beta}\colon S_\alpha\to S_\beta$ be a homomorphism, and assume that the following conditions hold:
\begin{enumerate}
  \item $\varphi_{\alpha,\alpha}\colon S_\alpha\to S_\alpha$ is the identity map of $S_\alpha$ for any $\alpha\in S_\alpha$;
  \item if $\gamma\preccurlyeq\beta\preccurlyeq\alpha$ in $E$, then $\varphi_{\alpha,\beta}\varphi_{\beta,\gamma}=\varphi_{\alpha,\gamma}$.
\end{enumerate}
On the set $S=\bigcup\limits_{\alpha\in E} S_{\alpha}$ define a semigroup operation by the formula
\begin{equation*}
  a*b=\left((a)\varphi_{\alpha,\alpha\beta}\right)\left((b)\varphi_{\beta,\alpha\beta}\right)
\end{equation*}
if $a\in S_\alpha$, $b\in S_\beta$, and denote $S=\left[E;S_\alpha,\varphi_{\alpha,\beta}\right]$.  The semigroup $\left[E;S_\alpha,\varphi_{\alpha,\beta}\right]$ is called a (\emph{strong}) \emph{se\-mi\-lattice of semigroups} $S_\alpha$.
\end{construction}

Well-known Clifford's Theorem state that an inverse semigroup $S$ is Clifford (i.e., $E(S)$ is contained in the center of $S$) if and only if $S$ is a semilattice of groups (see \cite[Theorem~II.2.6]{Petrich=1984}).

In \cite{Kochin=1968} Kochin showed that every simple inverse $\omega$-semigroup is isomorphic to the Bruck–Reilly extension $\mathbf{BR}(S,\theta)$ of a finite chain of groups $S=\left[E;G_\alpha,\varphi_{\alpha,\beta}\right]$.

A continuous map $f\colon X\to Y$ from a topological space $X$ into a topological space $Y$ is called:
\begin{itemize}
  \item[$\bullet$] \emph{quotient} if the set $(U)f^{-1}$ is open in $X$ if and only if $U$ is open in $Y$ (see \cite{Moore-1925} and \cite[Section~2.4]{Engelking=1989});
  \item[$\bullet$] \emph{hereditarily quotient} (or \emph{pseudoopen}) if for every $B\subset Y$ the restriction $f|_{B}\colon (B)f^{-1}\to B$ of $f$ is a quotient map (see \cite{McDougle-1958} and \cite[Section~2.4]{Engelking=1989});
  \item[$\bullet$] \emph{open} if $(U)f$ is open in $Y$ for every open subset $U$ in $X$;
  \item[$\bullet$] \emph{closed} if $(F)f$ is closed in $Y$ for every closed subset $F$ in $X$;
  \item[$\bullet$] \emph{perfect} if $X$ is Hausdorff, $f$ is a closed map and all fibers $(y)f^{-1}$ are compact subsets of $X$ \cite{Vainstein-1947}.
\end{itemize}
Every perfect map is closed, every closed map   and every hereditarily quotient map are quotient \cite{Engelking=1989}. Moreover, a continuous map $f\colon X\to Y$ from a topological space $X$ onto a topological space $Y$ is hereditarily quotient if and only if for every $y\in Y$ and every open subset $U$ in $X$ which contains $(y)f^{-1}$ we have that $y\in\operatorname{int}_Y(f(U))$ (see \cite[2.4.F]{Engelking=1989}).

A (\emph{semi})\emph{topological} \emph{semigroup} is a topological space with a (separately) continuous semigroup operation. An inverse topological semigroup with continuous inversion is called a \emph{topological inverse semigroup}.

\smallskip

A topology $\tau$ on a semigroup $S$ is called:
\begin{itemize}
  \item a \emph{semigroup} topology if $(S,\tau)$ is a topological semigroup;
  \item an \emph{inverse semigroup} topology if $(S,\tau)$ is a topological inverse semigroup;
  \item a \emph{shift-continuous} topology if $(S,\tau)$ is a semitopological semigroup;
  \item an \emph{inverse shift-continuous} topology if $(S,\tau)$ is a semitopological semigroup with continuous inversion.
\end{itemize}

We observe that if $S=\left[E;G_\alpha,\varphi_{\alpha,\beta}\right]$ is a semitopological Clifford semigroup then all homomorphisms $\varphi_{\alpha,\beta}$ are continuous \cite{Berglund=1973}.

It is well-known \cite{Eberhart-Selden=1969, Bertman-West=1976} that the bicyclic monoid ${\mathscr{C}}(p,q)$ admits only the discrete semigroup (shift-continuous) Hausdorff topology. Semigroup and shift-continuous $T_1$-topologies on ${\mathscr{C}}(p,q)$ are studied in \cite{Chornenka-Gutik=2023}. Topologizations of Bruck semigroups and Bruck--Reilly extensions, their topological properties and applications established in \cite{Gutik=1994, Gutik=1997, Gutik=2018, Gutik-Pavlyk=2009, Pavlyk=2008, Selden=1975, Selden=1976, Selden=1977}.

In the paper \cite{Gutik=2015} it is proved that every Hausdorff locally compact shift-continuous topology on the bicyclic monoid with adjoined zero is either compact or discrete. This result was extended by Bardyla onto the a polycyclic monoid \cite{Bardyla=2016} and graph inverse semigroups \cite{Bardyla=2018}, and by Mokrytskyi onto the monoid of order isomorphisms between principal filters of $\mathbb{N}^n$ with adjoined zero \cite{Mokrytskyi=2019}. In \cite{Bardyla=2023} Bardyla proved that a Hausdorff locally compact semitopological semigroup McAlister Semigroup $\mathcal{M}_1$ is either compact or discrete. However, this dichotomy does not hold for the McAlister Semigroup $\mathcal{M}_2$ and moreover, $\mathcal{M}_2$ admits continuum many different Hausdorff locally compact inverse semigroup topologies \cite{Bardyla=2023}.
Also, in \cite{Gutik-Maksymyk=2022} it is proved that the extended bicyclic semigroup $\mathscr{C}^0_{\mathbb{Z}}$ with adjoined zero admits distinct $\mathfrak{c}$-many shift-continuous topologies, however every Hausdorff locally compact semigroup topology on $\mathscr{C}^0_{\mathbb{Z}}$  is discrete.
 Algebraic properties on a group $G$ such that if the discrete group $G$ has these properties, then every locally compact shift continuous topology on $G$ with adjoined zero is either compact or discrete studied in \cite{Maksymyk=2019}.

In this paper we describe the structure of ($0$-)simple inverse Hausdorff semitopological $\omega$-semigroups with compact maximal subgroups. In particular, we show that if $S$ is a simple inverse Hausdorff semitopological $\omega$-semigroups with compact maximal subgroups, then $S$ is topologically isomorphic to the Bruck--Reilly extension $\left(\textbf{BR}(T,\theta),\tau_{\textbf{BR}}^{\oplus}\right)$ of a finite semilattice $T=\left[E;G_\alpha,\varphi_{\alpha,\beta}\right]$ of compact groups $G_\alpha$ in the class of topological inverse semigroups, where $\tau_{\textbf{BR}}^{\oplus}$ is the sum direct topology on $\textbf{BR}(T,\theta)$. Also we prove that every Hausdorff locally compact shift-continuous topology on the simple inverse Hausdorff semitopological $\omega$-semigroups with compact maximal subgroups with adjoined zero is either compact or  the zero is an isolated point.

\section{On simple inverse semitopological $\omega$-semigroups with compact maximal subgroups}\label{section-2}

Later we need the following simple lemma.

\begin{lemma}\label{lemma-2.1}
Let $S=\left[E;G_\alpha,\varphi_{\alpha,\beta}\right]$ be a semitopological semigroup which is a semilattice of groups $G_\alpha$. If $S$ is a topological sum of topological groups $G_\alpha$, then $S$ is a topological inverse semigroup.
\end{lemma}

\begin{proof}
Since $G_\alpha$ is a topological group for any $\alpha\in E$ and $S$ is a Clifford inverse semigroup, the inversion is continuous in $S$.

Fix arbitrary $a,b\in S$ such that $a\in G_\alpha$ and $b\in G_\beta$ for some $\alpha,\beta\in E$. The assumptions of the lemma imply that $G_{\gamma}$ is an open-and-closed subset of $S$ for any $\gamma\in E$. Since $G_{\alpha\beta}$ is a topological group, for any open neighbourhood  $U\left((a)\varphi_{\alpha,\alpha\beta}(b)\varphi_{\beta,\alpha\beta}\right)\subseteq G_{\alpha\beta}$ of the point $(a)\varphi_{\alpha,\alpha\beta}(b)\varphi_{\beta,\alpha\beta}$ in $S$ there exist open neighbourhoods $V\left((a)\varphi_{\alpha,\alpha\beta}\right)\subseteq G_{\alpha\beta}$ and  $V\left((b)\varphi_{\beta,\alpha\beta}\right)\subseteq G_{\alpha\beta}$ of the points $(a)\varphi_{\alpha,\alpha\beta}$ and $(b)\varphi_{\beta,\alpha\beta}$ in $S$, respectively, such that
\begin{equation*}
  V\left((a)\varphi_{\alpha,\alpha\beta}\right)\cdot V\left((b)\varphi_{\beta,\alpha\beta}\right)\subseteq U\left((a)\varphi_{\alpha,\alpha\beta}(b)\varphi_{\beta,\alpha\beta}\right).
\end{equation*}
Since homomorphisms $\varphi_{\alpha,\alpha\beta}\colon G_\alpha\to G_{\alpha\beta}$ and $\varphi_{\beta,\alpha\beta}\colon G_\beta G_{\alpha\beta}$ are continuous, and $G_{\gamma}$ is an open-and-closed subset of $S$ for any $\gamma\in E$, we have that there exist open neighbourhoods $W\left(a\right)\subseteq G_{\alpha}$ and  $W\left(b\right)\subseteq G_{\beta}$ of the points $a$ and $b$ in $S$, respectively, such that
\begin{equation*}
  (W\left(a\right))\varphi_{\alpha,\alpha\beta}\subseteq V\left((a)\varphi_{\alpha,\alpha\beta}\right) \qquad \hbox{and} \qquad  (W\left(b\right))\varphi_{\beta,\alpha\beta}\subseteq V\left((b)\varphi_{\beta,\alpha\beta}\right).
\end{equation*}
The above inclusions imply that
\begin{equation*}
  W\left(a\right)*W\left(b\right)\subseteq V\left((a)\varphi_{\alpha,\alpha\beta}\right)\cdot V\left((b)\varphi_{\beta,\alpha\beta}\right)\subseteq U\left((a)\varphi_{\alpha,\alpha\beta}(b)\varphi_{\beta,\alpha\beta}\right),
\end{equation*}
and hence, the semigroup operation in $S$ is continuous.
\end{proof}

\begin{proposition}\label{proposition-2.2}
Let $S=\left[E;G_\alpha,\varphi_{\alpha,\beta}\right]$ be a Hausdorff semitopological semigroup which is a finite semilattice of compact groups $G_\alpha$.  Then $S$ is a compact topological inverse semigroup.
\end{proposition}

\begin{proof}
Since the semilattice $E$ is finite, $S$ is a compact as the union of finitely many compact subsets $G_\alpha$. Also finiteness of $E$ and Hausdorffness of $S$ imply that $G_\alpha$ is open-and-closed subset of $S$. Next we apply Lemma~\ref{lemma-2.1}.
\end{proof}

\begin{definition}\label{definition-2.3}
Let $\mathfrak{STS}$ be a some class of semitopological semigroups and $(S,\tau_S)\in\mathfrak{STS}$. If $\tau_{\textbf{BR}}$ is a topology on $\textbf{BR}(S,\theta)$ such that $\left(\textbf{BR}(S,\theta),\tau_{\textbf{BR}}\right)\in\mathfrak{STS}$ and for some $i\in \omega$ the subsemigroup $S_{i,i}$ with the topology restricted from $\left(\textbf{BR}(S,\theta),\tau_{\textbf{BR}}\right)$ is topologically isomorphic to $(S,\tau_S)$ under the map $\xi_i\colon S_{i,i}\to S\colon (i,s,i)\mapsto s$, then $\left(\textbf{BR}(S,\theta),\tau_{\textbf{BR}}\right)$ is called a \emph{topological Bruck--Reilly extension} of $(S,\tau_S)$ in the class $\mathfrak{STS}$.
\end{definition}

\begin{proposition}\label{proposition-2.4}
Every Hausdorff semitopological simple inverse $\omega$-semigroup $S$ is topologically isomorphic to a topological Bruck–Reilly extension \emph{$\left(\textbf{BR}(T,\theta),\tau_{\textbf{BR}}\right)$} of a Hausdorff semitopological semigroup $T=\left[E;G_\alpha,\varphi_{\alpha,\beta}\right]$ which is a finite semilattice of semitiopological groups $G_\alpha$ in the class of semitopological semigroups. Moreover, if $S$ is locally compact, then $T$ is a locally compact semitopological semigroup.
\end{proposition}

\begin{proof}
By Kochin's Theorem (see \cite{Kochin=1968}) every simple inverse $\omega$-semigroup $S$ is (algebraically) isomorphic to the Bruck–Reilly extension of semigroup $T=\left[E;G_\alpha,\varphi_{\alpha,\beta}\right]$ which is a finite semilattice of groups $G_\alpha$. Then $T_{1,1}$ is a submonoid of $\textbf{BR}(T,\theta)$. Let $\tau_1$ be the topology induced from $\left(\textbf{BR}(T,\theta),\tau_{\textbf{BR}}\right)$ onto $T_{1,1}$. By Definition \ref{definition-2.3} the semitopological semigroup $\left(\textbf{BR}(T,\theta),\tau_{\textbf{BR}}\right)$ is a topological Bruck–Reilly extension of the semitopological semigroup $(T_{1,1},\tau_1)$. Moreover, by Proposition 2.4 of \cite{Gutik=2018} for any $i,j\in\omega$ the subsemigroups $T_{i,i}$ and $T_{j,j}$ with the induced from $\left(\textbf{BR}(T,\theta),\tau_{\textbf{BR}}\right)$ topologies  are topologically isomorphic by the mapping $f^{i,i}_{j,j}\colon T_{i,i}\to T_{j,j}$, $x\mapsto(j,1_S,i)\cdot x\cdot(i,1_S,j)$.

Also, Proposition 2.4 of \cite{Gutik=2018} implies that for any $i\in\omega$ the following sets $(i,1_S,i)\cdot \textbf{BR}(T,\theta)$ and $\textbf{BR}(T,\theta)\cdot (i,1_S,i)$ are retracts of $\left(\textbf{BR}(T,\theta),\tau_{\textbf{BR}}\right)$, and hence, by \cite[1.5.C]{Engelking=1989} they are closed subsets in the topological space $\left(\textbf{BR}(T,\theta),\tau_{\textbf{BR}}\right)$. Then
\begin{equation*}
T_{1,1}=\textbf{BR}(T,\theta)\setminus\left((1,1_S,1)\cdot \textbf{BR}(T,\theta)\cup \textbf{BR}(T,\theta)\cdot (1,1_S,1)\right)
\end{equation*}
is an open subset of $\left(\textbf{BR}(T,\theta),\tau_{\textbf{BR}}\right)$. This implies the last statement, because by Theorem 3.3.8 of \cite{Engelking=1989} an open subspace of a locally compact space is locally compact, as well.
\end{proof}

\begin{definition}\label{definition-2.5}
Let $\mathscr{B}_S$ be a base of the topology $\tau_S$ on a semitopological semigroup $S$. The topology $\tau_{\textbf{BR}}^{\oplus}$ on $\textbf{BR}(S,\theta)$ generated by the base $\mathscr{B}_{\textbf{BR}}^{\oplus}=\left\{U_{i,j}\colon U\in\mathscr{B}_S, \; i,j\in\omega\right\}$ is called a \emph{sum direct topology} on $\textbf{BR}(S,\theta)$.
\end{definition}

The following statement is proved in \cite{Gutik=1994} and \cite{Gutik-Pavlyk=2009}.

\begin{proposition}\label{proposition-2.6}
Let $(S,\tau_S)$ be a semitopological semigroup $S$. Then \emph{$\left(\textbf{BR}(S,\theta),\tau_{\textbf{BR}}^{\oplus}\right)$} is a semitopological semigroup, i.e., \emph{$\left(\textbf{BR}(S,\theta),\tau_{\textbf{BR}}^{\oplus}\right)$} is a topological Bruck--Reilly extension of $(S,\tau_S)$ in the class of semitopological semigroups. Moreover, if $(S,\tau_S)$ satisfies one of the following conditions: it is metrizable, Hausdorff, a semitopological semigroup with the continuous inversion, a topological semigroup, a topological inverse semigroup, then so is \emph{$\left(\textbf{BR}(S,\theta),\tau_{\textbf{BR}}^{\oplus}\right)$}, and \emph{$\left(\textbf{BR}(S,\theta),\tau_{\textbf{BR}}^{\oplus}\right)$} is a topological Bruck--Reilly extension of $(S,\tau_S)$ in the corresponding class of semitopological semigroups.
\end{proposition}

Theorem 8 of \cite{Gutik-Pavlyk=2009} implies the following

\begin{corollary}\label{corollary-2.7}
Let $(S,\tau_S)$ be a Hausdorff compact semitopological semigroup $S$. If \emph{$\left(\textbf{BR}(S,\theta),\tau_{\textbf{BR}}\right)$} is a topological Bruck--Reilly extension of $(S,\tau_S)$ in the class of Hausdorff semitopological semigroups, then \emph{$\tau_{\textbf{BR}}$} coincides with the sum direct topology \emph{$\tau_{\textbf{BR}}^{\oplus}$} on \emph{$\textbf{BR}(S,\theta)$}.
\end{corollary}

\begin{theorem}\label{theorem-2.8}
Let $T$ be a compact Hausdorff topological semigroup and \emph{$\left(\textbf{BR}(T,\theta),\tau_{\textbf{BR}}\right)$} be a topological Bruck--Reilly extension of $T$ in the class of Hausdorff semitopological semigroups. Then \emph{$\left(\textbf{BR}(T,\theta),\tau_{\textbf{BR}}\right)$} is a Hausdorff topological semigroup. Moreover, if $T$ is a topological inverse semigroup, then so is \emph{$\left(\textbf{BR}(T,\theta),\tau_{\textbf{BR}}\right)$}.
\end{theorem}

\begin{proof}
By Corollary~\ref{corollary-2.7}, $\tau_{\textbf{BR}}$ coincides with the sum direct topology $\tau_{\textbf{BR}}^{\oplus}$ on $\textbf{BR}(T,\theta)$.

Fix arbitrary $(i,s,j),(k,t,l)\in \textbf{BR}(T,\theta)$. Then we have that
\begin{equation*}
  (i,s,j)\cdot (k,t,l)=
  \left\{
    \begin{array}{ll}
      (i-j+k,(s)\theta^{k-j}\cdot t,l), & \hbox{if~} j<k;\\
      (i,s\cdot t,l),                   & \hbox{if~} j=k;\\
      (i,s\cdot (t)\theta^{j-k},j-k+l), & \hbox{if~} j>k.
    \end{array}
  \right.
\end{equation*}

Next we consider the following cases.

(1) Suppose that $j<k$. Then for any open neighbourhood $U((s)\theta^{k-j}\cdot t)$ of the point $(s)\theta^{k-j}\cdot t$ in $T$ there exist open neighbourhoods $V((s)\theta^{k-j})$ and $V(t)$ of the points $(s)\theta^{k-j}$ and $t$ in $T$, respectively, such that $V((s)\theta^{k-j})\cdot V(t)\subseteq U((s)\theta^{k-j}\cdot t)$, because $T$ is a topological semigroup. By Proposition 2.4 of \cite{Gutik=2018} the homomorphism $\theta\colon T\to H(1_T)$ is continuous, and hence there exists an open neighbourhood $O(s)$ of the point $s$ in $T$ such that $(O(s))\theta^{k-j}\subseteq V((s)\theta^{k-j})$. Since $j<k$, $O(s)_{i,j}\subseteq T_{i,j}$, $V(t)_{k,l}\subseteq T_{k,l}$, and $U((s)\theta^{k-j}\cdot t)_{i-j+k,l}\subseteq T_{i-j+k,l}$, the semigroup operation in $\textbf{BR}(T,\theta)$ implies that
\begin{equation*}
O(s)_{i,j}\cdot V(t)_{k,l}\subseteq U((s)\theta^{k-j}\cdot t)_{i-j+k,l}.
\end{equation*}

(2) Suppose that $j=k$. Since $T$ is a topological semigroup, for any open neighbourhood $U(s\cdot t)$ of the point $s\cdot t$ in the space $T$ there exist open neighbourhoods $V(s)$ and $V(t)$ of the points $s$ and $t$ in $T$, respectively, such that $V(s)\cdot V(t)\subseteq U(s\cdot t)$. Since $j=k$, $V(s)_{i,j}\subseteq T_{i,j}$, $V(t)_{k,l}\subseteq T_{k,l}$, and $U(s\cdot t)_{i,l}\subseteq T_{i,l}$, by the semigroup operation of $\textbf{BR}(T,\theta)$ we obtain that
\begin{equation*}
  V(s)_{i,j}\cdot V(t)_{k,l}\subseteq U(s\cdot t)_{i,l}.
\end{equation*}

(3) Suppose that $j>k$. Since $T$ is a topological semigroup, for any open neighbourhood $U(s\cdot (t)\theta^{j-k})$ of the point $s\cdot (t)\theta^{j-k}$ in the space $T$ there exist open neighbourhoods $V(s)$ and $V((t)\theta^{j-k})$ of the points $s$ and $(t)\theta^{j-k}$ in $T$, respectively, such that $V(s)\cdot  V((t)\theta^{j-k})\subseteq U(s\cdot (t)\theta^{j-k})$. By Proposition 2.4 of \cite{Gutik=2018} the homomorphism $\theta\colon T\to H(1_T)$ is continuous and hence, there exists an open neighbourhood $O(t)$ of the point $t$ in the topological space $T$ such that $(O(t))\theta^{j-k}\subseteq V((t)\theta^{j-k})$. Since $j>k$, $V(s)_{i,j}\subseteq T_{i,j}$, $O(t)_{k,l}\subseteq T_{k,l}$, and $U(s\cdot (t)\theta^{j-k})_{i,j-k+l}\subseteq T_{i,j-k+l}$,  by the semigroup operation of $\textbf{BR}(T,\theta)$ we get that
\begin{equation*}
  V(s)_{i,j}\cdot O(t)_{k,l}\subseteq U(s\cdot (t)\theta^{j-k})_{i,j-k+l}.
\end{equation*}

The above three cases imply that the semigroup operation is continuous in $\left(\textbf{BR}(T,\theta),\tau_{\textbf{BR}}\right)$.

If $T$ is an inverse semigroup, then $(i,s,j)^{-1}=(j,s^{-1},i)$ for any $(i,s,j)\in \textbf{BR}(T,\theta)$.  Since $T$ is an inverse topological semigroup, for any open neighbourhood $U(s^{-1})$ of the point $s^{-1}$ in $T$ there exists an open neighbourhood $V(s)$ of $s$ in $T$ such that $\left(V(s)\right)^{-1}\subseteq U(s^{-1})$. Since $V(s)_{i,j}\subseteq T_{i,j}$ and $U(s^{-1})_{j,i}\subseteq T_{j,i}$, the semigroup operation in $\textbf{BR}(T,\theta)$ implies that $(V(s)_{i,j})^{-1}\subseteq U(s^{-1})_{j,i}$, and hence, the inversion is continuous in $\left(\textbf{BR}(T,\theta),\tau_{\textbf{BR}}\right)$.
\end{proof}

The main result of this section is the following theorem.

\begin{theorem}\label{theorem-2.9}
Let $S$ be a Hausdorff semitopological simple inverse $\omega$-semigroup such that every maximal subgroup of $S$ is compact. Then $S$ is topologically isomorphic to the Bruck--Reilly extension \emph{$\left(\textbf{BR}(T,\theta),\tau_{\textbf{BR}}^{\oplus}\right)$} of a finite semilattice $T=\left[E;G_\alpha,\varphi_{\alpha,\beta}\right]$ of compact groups $G_\alpha$ in the class of topological inverse semigroups. Moreover, the space of $S$ is locally compact.
\end{theorem}

\begin{proof}
The first statement of the theorem follows from Proposition~\ref{proposition-2.4} and Theorem~\ref{theorem-2.8}. Theorem~3.3.12 of \cite{Engelking=1989} implies the second statement of the theorem.
\end{proof}

The following example shows that the statement of Theorem~\ref{theorem-2.9} is not true when a Hausdorff locally compact semitopological simple inverse $\omega$-semigroup $S$ contains non-compact maximal subgroup.

\begin{example}[\!{\cite[Example~4.7]{Gutik=2018}}]\label{example-2.10}
Let $\mathbb{Z}(+)$ be the additive group of integers and $0_{\mathbb{Z}}$ be the neutral element of  $\mathbb{Z}(+)$.
We define a topology $\tau_\mathrm{cf}$ on $\mathbf{BR}(\mathbb{Z}(+),\theta)$ in the following way. Let $(i,g,j)$ be an isolated point of $(\mathbf{BR}(\mathbb{Z}(+),\theta),\tau_\mathrm{cf})$ in the following cases:
\begin{enumerate}
  \item[$(i)$] $g\neq 0_{\mathbb{Z}}$ and $i, j\in\omega$;
  \item[$(ii)$] $i=0$ or $j=0$.
\end{enumerate}
The family
\begin{equation*}
  \mathscr{B}_{cf}(i,0_{\mathbb{Z}},j)=\left\{(UF)^0_{i-1,j-1}=(\mathbb{Z}(+)\setminus F)_{i-1,j-1}\cup \{(i,0_{\mathbb{Z}},j)\}\colon F \hbox{~is a fnite subset of~} \mathbb{Z}(+)\right\}
\end{equation*}
is a base of the topology $\tau_\mathrm{cf}$ on $\mathbf{BR}(\mathbb{Z}(+),\theta)$ at the point $(i,0_{\mathbb{Z}},j)$, for all $i,j\in\omega$. Then $(\mathbf{BR}(\mathbb{Z}(+),\theta),\tau_\mathrm{cf})$ is a Hausdorff locally compact semitopological inverse semigroup with continuous inversion.
\end{example}

\section{On adjoining zero to a simple inverse locally compact semitopological $\omega$-semigroup with compact maximal subgroups}\label{section-3}

Later in this section by $\textbf{BR}^{\mathbf{0}}(S,\theta)$ we denote the Bruck--Reilly semigroup $\textbf{BR}(S,\theta)$ with an adjoined zero $\mathbf{0}$ (see \cite[Section 1.1]{Clifford-Preston-1961}).

\begin{proposition}\label{proposition-3.1}
Let \emph{$\tau_{\textbf{BR}}^{\mathbf{0}}$} be a Hausdorff topology on \emph{$\textbf{BR}^{\mathbf{0}}(S,\theta)$} such that the set $S_{i,j}$ is open in \emph{$\left(\textbf{BR}^{\mathbf{0}}(S,\theta),\tau_{\textbf{BR}}^{\mathbf{0}}\right)$} for all $i,j\in \omega$. Then $\eta^\natural$ is a closed congruence on \emph{$\left(\textbf{BR}^{\mathbf{0}}(S,\theta),\tau_{\textbf{BR}}^{\mathbf{0}}\right)$}.
\end{proposition}

\begin{proof}
Fix arbitrary non-$\eta^\natural$-equivalent non-zero elements $(i,s,j)$ and $(m,t,n)$ of the  semigroup $\textbf{BR}^{\mathbf{0}}(S,\theta)$. Then $S_{i,j}$ and $S_{m,n}$ are open disjoint neighbourhoods of the points $(i,s,j)$ and $(m,t,n)$ in the space $\left(\textbf{BR}^{\mathbf{0}}(S,\theta),\tau_{\textbf{BR}}^{\mathbf{0}}\right)$, respectively, such that $\eta^\natural\cap \left(S_{i,j}\times S_{m,n}\right)=\varnothing$. Since the topology $\tau_{\textbf{BR}}^{\mathbf{0}}$ is Hausdorff, there exist disjoint open neighbourhoods $U(i,s,j)$ and $U(\mathbf{0})$ of $(i,s,j)$ and $\mathbf{0}$ in $\left(\textbf{BR}^{\mathbf{0}}(S,\theta),\tau_{\textbf{BR}}^{\mathbf{0}}\right)$, respectively. This implies that $U(i,s,j)\times U(\mathbf{0})$ is an open neighbourhood of the ordered pair $\left((i, s, j),\mathbf{0}\right)$ in $\textbf{BR}^{\mathbf{0}}(S,\theta)\times \textbf{BR}^{\mathbf{0}}(S,\theta)$ with the product topology which does not intersect the congruence $\eta^\natural$ of the semigroup $\textbf{BR}^{\mathbf{0}}(S,\theta)$. Hence, $\eta^\natural$ is a closed congruence on the semigroup $\left(\textbf{BR}^{\mathbf{0}}(S,\theta),\tau_{\textbf{BR}}^{\mathbf{0}}\right)$.
\end{proof}

We put $\mathscr{C}^0=\mathscr{C}(p,q)\sqcup\{0\}$ be the bicyclic semigroup with adjoined zero. Obviously that the congruence $\eta^\natural$ on the Bruck--Reilly extension $\textbf{BR}^{\mathbf{0}}(S,\theta)$ of a semigroup $S$ generates the natural homomorphism $\eta\colon \textbf{BR}^{\mathbf{0}}(S,\theta)\to \mathscr{C}^0$.

\begin{lemma}\label{lemma-3.2}
Let \emph{$\left(\textbf{BR}^{\mathbf{0}}(S,\theta),\tau_{\textbf{BR}}^{\mathbf{0}}\right)$} be a semitopological semigroup with a compact (left, right) ideal. If the natural homomorphism \emph{$\eta\colon \textbf{BR}^{\mathbf{0}}(S,\theta)\to \mathscr{C}^0$} is a quotient map, then $\eta$ is an open map.
\end{lemma}

\begin{proof}
Suppose that on $\mathscr{C}^0$ admits a topology such that the natural homomorphism $\eta\colon \textbf{BR}^{\mathbf{0}}(S,\theta)\to \mathscr{C}^0$ is a quotient map.

If $U$ is an open subset of $\left(\textbf{BR}^{\mathbf{0}}(S,\theta),\tau_{\textbf{BR}}^{\mathbf{0}}\right)$ such that $U\not\ni \mathbf{0}$, then $\eta(U)$ is an open subset of $\mathscr{C}^0$, because by Proposition 1 of \cite{Eberhart-Selden=1969}  the bicyclic monoid $\mathscr{C}(p,q)$ is a discrete open subset of the space $\mathscr{C}^0$.

Suppose $U\ni\mathbf{0}$ is an open subset of $\left(\textbf{BR}^{\mathbf{0}}(S,\theta),\tau_{\textbf{BR}}^{\mathbf{0}}\right)$. Put $U^*=\eta^{-1}(\eta(U))$. Then $U^*=\eta^{-1}(\eta(U^*))$. Since $\eta\colon \textbf{BR}^{\mathbf{0}}(S,\theta)\to \mathscr{C}^0$ is a natural homomorphism, $U^*=\bigcup\left\{ G_{i,j}\colon G_{i,j}\cap U\neq\varnothing\right\}\cup\{\mathbf{0}\}$. By Theorem~8 of \cite{Gutik-Pavlyk=2009} the restriction of the topology $\tau_{\textbf{BR}}^{\mathbf{0}}$ on the semigroup $\textbf{BR}(S,\theta)$ coincides with the sum direct topology $\tau_{\textbf{BR}}^{\oplus}$ on $\textbf{BR}(S,\theta)$. This implies that  $U^*$ is an open subset of the space $\left(\textbf{BR}^{\mathbf{0}}(S,\theta),\tau_{\textbf{BR}}^{\mathbf{0}}\right)$, and since $\eta$ is a quotient map and $U^*=\eta^{-1}(\eta(U^*))$, we conclude that $\eta(U)$ is an open subset of the space $\mathscr{C}^0$.
\end{proof}

The following example from \cite{Gutik=2015} shows that the semigroup $\mathscr{C}^0$ admits a shift-continuous compact Hausdorff topology.

\begin{example}[\cite{Gutik=2015}]\label{example-3.3}
On the semigroup $\mathscr{C}^0$ we define a topology $\tau_{\operatorname{\textsf{Ac}}}$ in the following way:
\begin{itemize}
  \item[$(i)$] every element of the bicyclic monoid ${\mathscr{C}}(p,q)$ is an isolated point in the space $(\mathscr{C}^0,\tau_{\operatorname{\textsf{Ac}}})$;
  \item[$(ii)$] the family $\mathscr{B}(0)=\left\{U\subseteq \mathscr{C}^0\colon U\ni 0 \hbox{~and~} {\mathscr{C}}(p,q)\setminus U \hbox{~is finite}\right\}$ determines a base of the topology $\tau_{\operatorname{\textsf{Ac}}}$ at zero $0\in\mathscr{C}^0$,
\end{itemize}
i.e., $\tau_{\operatorname{\textsf{Ac}}}$ is the topology of the Alexandroff one-point compactification of the discrete space ${\mathscr{C}}(p,q)$ with the remainder $\{0\}$. Then $(\mathscr{C}^0,\tau_{\operatorname{\textsf{Ac}}})$ is a Hausdorff compact semitopological semigroup.
\end{example}

\begin{lemma}\label{lemma-3.4}
Let \emph{$\left(\textbf{BR}^{\mathbf{0}}(S,\theta),\tau_{\textbf{BR}}^{\mathbf{0}}\right)$} be a Hausdorff semitopological semigroup with a compact subsemigroup $S_{i,i}$ for some $i\in\omega$. Then  $S_{i,j}$  is an open-and-closed subspace of \emph{$\left(\textbf{BR}^{\mathbf{0}}(S,\theta),\tau_{\textbf{BR}}^{\mathbf{0}}\right)$} for any $i,j\in\omega$.
\end{lemma}

\begin{proof}
Since $(i,1_S,i)$ is an idempotent of $\textbf{BR}^{\mathbf{0}}(S,\theta)$ for any $i\in\omega$ the subsets
$(i,1_S,i)\cdot \textbf{BR}^{\mathbf{0}}(T,\theta)$ and $\textbf{BR}^{\mathbf{0}}(T,\theta)\cdot (i,1_S,i)$ are retracts of $\left(\textbf{BR}^{\mathbf{0}}(T,\theta),\tau_{\textbf{BR}}^{\mathbf{0}}\right)$, and hence by \cite[1.5.C]{Engelking=1989} they are closed subsets in the topological space $\left(\textbf{BR}^{\mathbf{0}}(T,\theta),\tau_{\textbf{BR}}^{\mathbf{0}}\right)$. Then
\begin{equation*}
T_{k,k}=\textbf{BR}^{\mathbf{0}}(T,\theta)\setminus\left((k+1,1_S,k+1)\cdot \textbf{BR}^{\mathbf{0}}(T,\theta)\cup \textbf{BR}(T,\theta)\cdot (k+1,1_S,k+1)\right)
\end{equation*}
is an open subset of $\left(\textbf{BR}(T,\theta),\tau_{\textbf{BR}}\right)$ for any  $k\in\omega$. Since the subsemigroup $S_{i,i}$ of $\left(\textbf{BR}^{\mathbf{0}}(S,\theta),\tau_{\textbf{BR}}^{\mathbf{0}}\right)$ is compact for some $i\in\omega$ and by Proposition~2.4$(iv)$ from \cite{Gutik=2018} the subspaces $S_{i,j}$ of $\left(\textbf{BR}^{\mathbf{0}}(S,\theta),\tau_{\textbf{BR}}^{\mathbf{0}}\right)$, $i,j\in\omega$, are hemeomorphic, $S_{i,j}$ are compact subspaces of $\left(\textbf{BR}^{\mathbf{0}}(S,\theta),\tau_{\textbf{BR}}^{\mathbf{0}}\right)$. Then for any $i,j\leqslant k$ the subspace $S_{i,j}$ is open-and-closed in $\left(\textbf{BR}^{\mathbf{0}}(S,\theta),\tau_{\textbf{BR}}^{\mathbf{0}}\right)$.
\end{proof}

\begin{proposition}\label{proposition-3.5}
Let \emph{$\left(\textbf{BR}^{\mathbf{0}}(S,\theta),\tau_{\textbf{BR}}^{\mathbf{0}}\right)$} be a Hausdorff locally compact semitopological semigroup with a compact subsemigroup $S_{i,i}$ for some $i\in\omega$. Then the quotient semigroup \emph{$\textbf{BR}^{\mathbf{0}}(S,\theta)/\eta^\natural$} with the quotient topology is topologically isomorphic to the semigroup $\mathscr{C}^0$ with either the topology $\tau_{\operatorname{\textsf{Ac}}}$ or the discrete topology.
\end{proposition}

\begin{proof}
By Lemma~\ref{lemma-3.4}, $S_{i,j}$  is an open-and-closed subspace of \emph{$\left(\textbf{BR}^{\mathbf{0}}(S,\theta),\tau_{\textbf{BR}}^{\mathbf{0}}\right)$} for any $i,j\in\omega$ and hence, by Proposition~\ref{proposition-3.1}, $\eta^\natural$ is a closed congruence on $\left(\textbf{BR}^{\mathbf{0}}(S,\theta),\tau_{\textbf{BR}}^{\mathbf{0}}\right)$. Then the quotient semigroup $\textbf{BR}^{\mathbf{0}}(S,\theta)/\eta^\natural$ with the quotient topology is a Hausdorff space. Lemma~\ref{lemma-3.2} implies that $\eta\colon \textbf{BR}^{\mathbf{0}}(S,\theta)\to \mathscr{C}^0$ is an open map. Hence by Theorem 3.3.15 from \cite{Engelking=1989}, the quotient semigroup $\textbf{BR}^{\mathbf{0}}(S,\theta)/\eta^\natural$ with the quotient topology is a locally compact space. Since $\textbf{BR}^{\mathbf{0}}(S,\theta)/\eta^\natural$ is isomorphic to the semigroup $\mathscr{C}^0$, Theorem 1 of \cite{Gutik=2015} implies the statement of the proposition.
\end{proof}

Later in this section, if the otherwise is not stated explicitly, we assume that $\tau_{\textbf{BR}}^{\mathbf{0}}$ is a Hausdorff locally compact shift-continuous topology on the semigroup $\textbf{BR}^{\mathbf{0}}(S,\theta)$ such that the following conditions hold:
\begin{enumerate}
  \item[$(i)$] the subsemigroup $S_{i,i}$ of $\textbf{BR}^{\mathbf{0}}(S,\theta)$ with the restriction topology from $\left(\textbf{BR}^{\mathbf{0}}(S,\theta),\tau_{\textbf{BR}}^{\mathbf{0}}\right)$ is a compact semitopological semigroup for some $i\in\omega$ (and hence, by Proposition~2.4 of \cite{Gutik=2018} for all $i\in\omega$);
  \item[$(ii)$] $\mathbf{0}$ is non-isolated point of $\left(\textbf{BR}^{\mathbf{0}}(S,\theta),\tau_{\textbf{BR}}^{\mathbf{0}}\right)$.
\end{enumerate}

Let $\mathscr{P}=\{P_\alpha\colon \alpha\in\mathscr{I}\}$ be an infinite family of nonempty subsets of a set $X$. We shall say that a set $A\subseteq X$ \emph{intersects almost all} subsets of $\mathscr{P}$ if $A\cap P_\alpha=\varnothing$ for finitely many $P_\alpha\in \mathscr{P}$.

\begin{lemma}\label{lemma-3.6}
Every open neighbourhood $U_{\mathbf{0}}$ of zero $\mathbf{0}$ in \emph{$\left(\textbf{BR}^{\mathbf{0}}(S,\theta),\tau_{\textbf{BR}}^{\mathbf{0}}\right)$} intersects almost all subsets $S_{i,j}$, $i,j\in\omega$, of \emph{$\textbf{BR}(S,\theta)$}
\end{lemma}

\begin{proof}
Suppose to the contrary that there exists an open neighbourhood $U_{\mathbf{0}}$ of zero in $\left(\textbf{BR}^{\mathbf{0}}(S,\theta),\tau_{\textbf{BR}}^{\mathbf{0}}\right)$ such that $U_{\mathbf{0}}\cap S_{i,j}=\varnothing$ for infinitely many $S_{i,j}$, $i,j\in\omega$. Then by Lemma~\ref{lemma-3.2} the quotient natural homomorphism $\eta\colon \textbf{BR}^{\mathbf{0}}(S,\theta)\to \mathscr{C}^0$ is an open map, and hence the quotient semigroup $\textbf{BR}^{\mathbf{0}}(S,\theta)/\eta^\natural$  with the quotient topology is neither compact nor discrete, which contradicts Proposition~\ref{proposition-3.5}.
\end{proof}

For an arbitrary subset $A$ of $\textbf{BR}^{\mathbf{0}}(S,\theta)$ and any $i,j\in\omega$ we denote $[A]_{i,j}=A\cap S_{i,j}$.

\begin{lemma}\label{lemma-3.7}
For every open neighbourhood $U_{\mathbf{0}}$ of zero $\mathbf{0}$ in \emph{$\left(\textbf{BR}^{\mathbf{0}}(S,\theta),\tau_{\textbf{BR}}^{\mathbf{0}}\right)$} and any $i_0\in\omega$ the sets
\begin{equation*}
  \left\{j\in\omega\colon S_{i_0,j}\nsubseteq  U_{\mathbf{0}}\right\} \qquad \hbox{and} \qquad \left\{j\in\omega\colon S_{j,i_0}\nsubseteq  U_{\mathbf{0}}\right\}
\end{equation*}
are finite.
\end{lemma}

\begin{proof}
Suppose to the contrary that there exist an open neighbourhood $U_{\mathbf{0}}$ of zero in $\left(\textbf{BR}^{\mathbf{0}}(S,\theta),\tau_{\textbf{BR}}^{\mathbf{0}}\right)$ and $i_0\in\omega$ such that the set $\left\{j\in\omega\colon S_{i_0,j}\subseteq  U_{\mathbf{0}}\right\}$ is infinite. Since $\left(\textbf{BR}^{\mathbf{0}}(S,\theta),\tau_{\textbf{BR}}^{\mathbf{0}}\right)$ is a locally compact space, we can take a regular open neighbourhood $U_{\mathbf{0}}$ of the zero with compact closure.

We consider the following two cases:
\begin{enumerate}
  \item[$(i)$] there exists $j_0\in\omega$ such that $[U_{\mathbf{0}}]_{i_0,j}\neq S_{i_0,j}$ for all $j\geqslant j_0$;
  \item[$(ii)$] for every $k\in\mathbb{N}$ there exists a positive integer $n>k$ such that  $[U_{\mathbf{0}}]_{i_0,n}= S_{i_0,n}$.
\end{enumerate}

Suppose case $(i)$ holds. Since every subset $S_{i,j}$, $i,j\in\omega$, of $\left(\textbf{BR}^{\mathbf{0}}(S,\theta),\tau_{\textbf{BR}}^{\mathbf{0}}\right)$ is compact, the separate continuity of the semigroup operation in $\left(\textbf{BR}^{\mathbf{0}}(S,\theta),\tau_{\textbf{BR}}^{\mathbf{0}}\right)$ and Lemma~\ref{lemma-3.6} imply that without loss of generality we may assume that $j_0=0$. By Lemma~\ref{lemma-3.4} every subset $S_{i,j}$ is open-and-compact in $\left(\textbf{BR}^{\mathbf{0}}(S,\theta),\tau_{\textbf{BR}}^{\mathbf{0}}\right)$, and hence, the set
\begin{equation*}
  \mathcal{S}^0_{i_0}(U_{\mathbf{0}})=\left\{\mathbf{0}\right\}\cup \bigcup_{j\in\omega} \left[\mathrm{cl}_{\textbf{BR}^{\mathbf{0}}(S,\theta)}(U_{\mathbf{0}})\right]_{i_0,j}
\end{equation*}
is compact. By Lemma~\ref{lemma-3.4} the family $\mathscr{U}_{i_0}=\left\{\left\{U_{\mathbf{0}}\right\},\left\{S_{i_0,j}\colon j\in\omega\right\}\right\}$ is an open cover of the compactum $\mathcal{S}^0_{i_0}(U_{\mathbf{0}})$, and hence, there exists $j_1\in\mathbb{N}$ such that
\begin{equation*}
  \left[U_{\mathbf{0}}\right]_{i_0,n}=\left[\mathrm{cl}_{\textbf{BR}^{\mathbf{0}}(S,\theta)}(U_{\mathbf{0}})\right]_{i_0,n}
\end{equation*}
for all  integers $n\geqslant j_1$. Since the right shift
\begin{equation*}
  \rho_{(1,1_S,0)}\colon \textbf{BR}^{\mathbf{0}}(S,\theta)\to \textbf{BR}^{\mathbf{0}}(S,\theta), \; x\mapsto x\cdot (1,1_S,0)
\end{equation*}
is continuous in $\left(\textbf{BR}^{\mathbf{0}}(S,\theta),\tau_{\textbf{BR}}^{\mathbf{0}}\right)$, the full preimage $V_{\mathbf{0}}=\rho^{-1}_{(1,1_S,0)}(U_{\mathbf{0}})$ is an open neighbourhood of the zero $\mathbf{0}$. By Lemma~\ref{lemma-3.4} the family $\mathscr{V}_{i_0}=\left\{\left\{V_{\mathbf{0}}\right\},\left\{S_{i_0,j}\colon j\in\omega\right\}\right\}$ is an open cover of the compactum $\mathcal{S}^0_{i_0}(U_{\mathbf{0}})$, and hence, there exists a positive integer $j_2\geqslant j_1$ such that
\begin{equation}\label{eq-3.1}
  \left[V_{\mathbf{0}}\right]_{i_0,n}=\left[U_{\mathbf{0}}\right]_{i_0,n}= \left[\mathrm{cl}_{\textbf{BR}^{\mathbf{0}}(S,\theta)}(U_{\mathbf{0}})\right]_{i_0,n}
\end{equation}
for all integers $n\geqslant j_2$. Indeed, since $(i_0,s,j)\cdot(1,1_S,0)=(i_0,s,j-1)$ for all $j\in\mathbb{N}$ and any $s\in S$, we obtain that the equalities \eqref{eq-3.1} holds for all integers $n\geqslant j_2$. Put $\widetilde{U}_{\mathbf{0}}=U_{\mathbf{0}}\setminus\left(S_{i_0,0}\cup\cdots\cup S_{i_0,j_2-1}\right)$. By Lemma~\ref{lemma-3.4}, $\widetilde{U}_{\mathbf{0}}$ is an open neighbourhood of zero in $\left(\textbf{BR}^{\mathbf{0}}(S,\theta),\tau_{\textbf{BR}}^{\mathbf{0}}\right)$ such that
\begin{equation*}
  \big[\widetilde{U}_{\mathbf{0}}\big]_{i_0,n}=\left[U_{\mathbf{0}}\right]_{i_0,n}= \left[\mathrm{cl}_{\textbf{BR}^{\mathbf{0}}(S,\theta)}(U_{\mathbf{0}})\right]_{i_0,n}
\end{equation*}
for all integers $n\geqslant j_2$. Since the space $\left(\textbf{BR}^{\mathbf{0}}(S,\theta),\tau_{\textbf{BR}}^{\mathbf{0}}\right)$ is locally compact, without loss of generality we may assume that the neighbourhood $U_{\mathbf{0}}$ is a regular open set. This implies that $\widetilde{U}_{\mathbf{0}}$ is a regular open set, as well. Hence  there exist distinct $s,t\in S$ such that $(i_0,s,n)\notin \left[U_{\mathbf{0}}\right]_{i_0,n}$ and $(i_0,t,n)\in\left[U_{\mathbf{0}}\right]_{i_0,n}$ for all integers $n\geqslant j_2$. But we have that
\begin{equation*}
  (i_0,s\cdot ((t)\theta)^{-1},i_0+1)\cdot (i_0,t,n)=(i_0,s\cdot ((t)\theta)^{-1}\cdot(t)\theta,n+1)=(i_0,s,n+1).
\end{equation*}
Let $W_{\mathbf{0}}=(\widetilde{U}_{\mathbf{0}})\lambda^{-1}_{(i_0,s\cdot ((t)\theta)^{-1},i_0+1)}$, where $\lambda_{(i_0,s\cdot ((t)\theta)^{-1},i_0+1)}$ is the left shift on the element  $(i_0,s\cdot ((t)\theta)^{-1},i_0+1)$ in the semigroup $\textbf{BR}^{\mathbf{0}}(S,\theta)$. Then  we have that
\begin{equation*}
  \big[\widetilde{U}_{\mathbf{0}}\big]_{i_0,n}\setminus \big[W_{\mathbf{0}}\big]_{i_0,n}\neq\varnothing \qquad \hbox{and} \qquad
  \big[W_{\mathbf{0}}\big]_{i_0,n}\setminus \big[\widetilde{U}_{\mathbf{0}}\big]_{i_0,n}\neq\varnothing
\end{equation*}
for all integers $n\geqslant j_2+1$. By Lemma~\ref{lemma-3.4}  the family $\mathscr{W}_{i_0}=\left\{\left\{W_0\right\}, \left\{S_{i_0,j}\colon j\in\omega\right\}\right\}$  is an open cover of $\mathcal{S}^0_{i_0}(U_{\mathbf{0}})$ which has no a finite subcover. This contradicts the
compactness of $\mathcal{S}^0_{i_0}(U_{\mathbf{0}})$, and hence, the set $\left\{j\in\omega\colon S_{i_0,j}\nsubseteq  U_{\mathbf{0}}\right\}$ is finite.

Suppose case $(ii)$ holds. Then there are infinitely many $j\in\omega$ such that $[U_{\mathbf{0}}]_{i_0,j}= S_{i_0,j}$ but $[U_{\mathbf{0}}]_{i_0,j-1}\neq S_{i_0,j-1}$. Since every subset $S_{i,j}$, $i,j\in\omega$, of $\left(\textbf{BR}^{\mathbf{0}}(S,\theta),\tau_{\textbf{BR}}^{\mathbf{0}}\right)$ is compact, Lemma~\ref{lemma-3.4} implies that every subset $S_{i,j}$ is open-and-compact in $\left(\textbf{BR}^{\mathbf{0}}(S,\theta),\tau_{\textbf{BR}}^{\mathbf{0}}\right)$, and hence, the set
\begin{equation*}
  \mathcal{S}^0_{i_0}(U_{\mathbf{0}})=\left\{\mathbf{0}\right\}\cup \bigcup_{j\in\omega} \left[\mathrm{cl}_{\textbf{BR}^{\mathbf{0}}(S,\theta)}(U_{\mathbf{0}})\right]_{i_0,j}
\end{equation*}
is compact. Let $V_{\mathbf{0}}=(U_{\mathbf{0}})\rho^{-1}_{(1,1_S,0)}$, where $\rho_{(1,1_S,0)}$ is the right shift on the element $(1,1_S,0)$ in the semigroup $\textbf{BR}^{\mathbf{0}}(S,\theta)$. By Lemma~\ref{lemma-3.4} the family $\mathscr{V}_{i_0}=\left\{\left\{V_0\right\}, \left\{S_{i_0,j}\colon j\in\omega\right\}\right\}$  is an open cover of the compactum $\mathcal{S}^0_{i_0}(U_{\mathbf{0}})$.  Then the continuity of the right shift $\rho_{(1,1_S,0)}$ in $\left(\textbf{BR}^{\mathbf{0}}(S,\theta),\tau_{\textbf{BR}}^{\mathbf{0}}\right)$ and the equality $(i_0,s,j)\cdot(1,1_S,0)=(i_0,s,j-1)$ imply that $\big[V_{\mathbf{0}}\big]_{i_0,j}\neq S_{i_0,j}$ for infnitely many $j\in\omega$. Also, the equality $(i_0,s,j)\cdot(1,1_S,0)=(i_0,s,j-1)$ and assumption of case $(ii)$ imply that $\big[U_{\mathbf{0}}\big]_{i_0,j}\setminus \big[V_{\mathbf{0}}\big]_{i_0,j}\neq\varnothing$ for infinitely many $j\in\omega$. Hence, the open cover $\mathscr{V}_{i_0}$ of $\mathcal{S}^0_{i_0}(U_{\mathbf{0}})$ does not have finite subcovers, which contradicts the compactness of $\mathcal{S}^0_{i_0}(U_{\mathbf{0}})$ and hence, the set $\left\{j\in\omega\colon S_{i_0,j}\nsubseteq  U_{\mathbf{0}}\right\}$ is finite.

The proof of the statement that the set $\left\{j\in\omega\colon S_{j,i_0}\nsubseteq  U_{\mathbf{0}}\right\}$ is finite is similar.
\end{proof}

\begin{lemma}\label{lemma-3.8}
For every open neighbourhood $U_{\mathbf{0}}$ of zero $\mathbf{0}$ in \emph{$\left(\textbf{BR}^{\mathbf{0}}(S,\theta),\tau_{\textbf{BR}}^{\mathbf{0}}\right)$} the set
\begin{equation*}
N_{U_{\mathbf{0}}}=\left\{(i, j)\in\omega\times\omega\colon S_{i,j}\subseteq U_{\mathbf{0}}\right\}
\end{equation*}
is finite.
\end{lemma}

\begin{proof}
Suppose to the contrary that there exists an open neighbourhood $U_{\mathbf{0}}$ of zero $\mathbf{0}$ in $\left(\textbf{BR}^{\mathbf{0}}(S,\theta),\tau_{\textbf{BR}}^{\mathbf{0}}\right)$ such that the set $N_{U_{\mathbf{0}}}$ is infinite. Since $\left(\textbf{BR}^{\mathbf{0}}(S,\theta),\tau_{\textbf{BR}}^{\mathbf{0}}\right)$ is a locally compact space, without loss of generality we may assume that the closure $\mathrm{cl}_{\textbf{BR}^{\mathbf{0}}(S,\theta)}(U_{\mathbf{0}})$ of the neighbourhood $U_{\mathbf{0}}$  is compact and the neighbourhood $U_{\mathbf{0}}$  is regular open. By Lemma~\ref{lemma-3.7} for every $k\in \mathbb{N}$ there exists $(i,j)\in N_{U_{\mathbf{0}}}$ such that $i>k$ and $j>k$.

Using induction we define an infinite sequence $\{(i_n,j_n)\}_{n\in\omega}$ of elements of the set $N_{U_{\mathbf{0}}}$ in the following way. By the assumption, there exists the smallest $i_0\in\omega$ such that $S_{i_0,j}\nsubseteq U_{\mathbf{0}}$, $j\in\omega$. By Lemma~\ref{lemma-3.7} there exits $j_0=\max\left\{j\in\omega\colon S_{i_0,j}\nsubseteq U_{\mathbf{0}}\right\}$.

At $(k+1)$-th step of induction we define pair $(i_{k+1,jk+1})\in N_{U_{\mathbf{0}}}$ as follows. Let $i_{k+1}$ be the smallest integer which is greater than $i_k$ such that $S_{i_k,j}\nsubseteq U_{\mathbf{0}}$, $j\in\omega$. By Lemma~\ref{lemma-3.7} there exits $j_{k+1}=\max\left\{j\in\omega\colon S_{i_{k+1},j}\nsubseteq U_{\mathbf{0}}\right\}$. Our assumption and Lemma~\ref{lemma-3.7} imply that the ordered pair $(i_{k+1},j_{k+1})$ belongs to $N_{U_{\mathbf{0}}}$.

By the separate continuity of the semigroup operation in $\left(\textbf{BR}^{\mathbf{0}}(S,\theta),\tau_{\textbf{BR}}^{\mathbf{0}}\right)$ there exists an open neighbourhood $V_{\mathbf{0}}\subseteq U_{\mathbf{0}}$ of zero in $\left(\textbf{BR}^{\mathbf{0}}(S,\theta),\tau_{\textbf{BR}}^{\mathbf{0}}\right)$ such that $V_{\mathbf{0}}\cdot (1,1_S,0)\subseteq U_{\mathbf{0}}$. The construction of the sequence $\{(i_n,j_n)\}_{n\in\omega}$ implies that
\begin{equation*}
  [V_{\mathbf{0}}]_{i_n,j_n}\subseteq [U_{\mathbf{0}}]_{i_n,j_n}\neq S_{i_n,j_n} \qquad \hbox{and} \qquad [U_{\mathbf{0}}]_{i_n,j_n+1}= S_{i_n,j_n+1}
\end{equation*}
for each $(i_n,j_n)\in N_{U_{\mathbf{0}}}$. By Lemma~\ref{lemma-3.4} the family $\mathscr{V}=\left\{\left\{V_0\right\}, \left\{S_{i,j}\colon i,j\in\omega\right\}\right\}$ is an open cover of the compact set $\mathrm{cl}_{\textbf{BR}^{\mathbf{0}}(S,\theta)}(U_{\mathbf{0}})$. The
continuity of the right shift $\rho_{(1,1_S,0)}$ in $\left(\textbf{BR}^{\mathbf{0}}(S,\theta),\tau_{\textbf{BR}}^{\mathbf{0}}\right)$ implies that $[V_{\mathbf{0}}]_{i_n,j_n+1}\neq S_{i_n,j_n+1}$ for infinitely many ordered pairs $(i_n,j_{n+1})\in N_{U_{\mathbf{0}}}$. Hence, we obtain that $[U_{\mathbf{0}}]_{i_n,j_n+1}\setminus [V_{\mathbf{0}}]_{i_n,j_n+1}\neq\varnothing$ for infinitely many $(i_n,j_{n+1})\in N_{U_{\mathbf{0}}}$. The above arguments guarantee that the cover $\mathscr{V}$ has no fnite subcovers, which contradicts the compactness of $\mathrm{cl}_{\textbf{BR}^{\mathbf{0}}(S,\theta)}(U_{\mathbf{0}})$. The obtained contradiction implies the statement of the lemma.
\end{proof}

\begin{example}\label{example-3.9}
Let $(S,\tau_S)$ be a Hausdorff semitopological monoid, $\theta\colon S\to H(1_S)$ be a continuous homomorphism and $\mathscr{B}_S(s)$ be a base of the topology $\tau_S$ at a point $s\in S$.

On the semigroup $\textbf{BR}^{\mathbf{0}}(S,\theta)$ we define a topology $\tau^\oplus_{\textbf{BR}}$ in the following
way:
\begin{enumerate}
  \item[$(i)$] for any non-zero element $(i,s,j)\in S_{i,j}$ of the semigroup $\textbf{BR}^{\mathbf{0}}(S,\theta)$ the family
  \begin{equation*}
    \mathscr{B}^\oplus_{\textbf{BR}}(i,s,j)=\left\{U_{i,j}\colon U\in\mathscr{B}_S(s) \right\}
  \end{equation*}
  is a base of the topology $\tau^\oplus_{\textbf{BR}}$ at the point $(i,s,j)\in \textbf{BR}^{\mathbf{0}}(S,\theta)$;

  \item[$(ii)$]  zero $\mathbf{0}\in \textbf{BR}^{\mathbf{0}}(S,\theta)$ is an isolated point in $\big(\textbf{BR}^{\mathbf{0}}(S,\theta),\tau^\oplus_{\textbf{BR}}\big)$.
\end{enumerate}
The semigroup operation in $\big(\textbf{BR}^{\mathbf{0}}(S,\theta),\tau^\oplus_{\textbf{BR}}\big)$ is separately continuous (see \cite{Gutik-Pavlyk=2009}). Moreover, if $(S,\tau_S)$ be a topological monoid, then so is $\big(\textbf{BR}^{\mathbf{0}}(S,\theta),\tau^\oplus_{\textbf{BR}}\big)$ \cite{Gutik=1994}.
\end{example}

In Example~\ref{example-3.10}  we extend the construction, proposed in Example 3.4 from \cite{Gutik=2018}, onto compact Bruck--Reilly extensions of compact semitopological monoids in the class of Hausdorff semitopological semigroups with adjoined zero.

\begin{example}\label{example-3.10}
Let $(S,\tau_S)$ be a Hausdorff compact semitopological monoid, $\theta\colon S\to H(1_S)$ be a continuous homomorphism and $\mathscr{B}_S(s)$ be a base of the topology $\tau_S$ at a point $s\in S$. On the semigroup $\textbf{BR}^{\mathbf{0}}(S,\theta)$ we define a topology $\tau^{\mathbf{Ac}}_{\textbf{BR}}$ in the following way:
\begin{enumerate}
  \item[$(i)$] for any non-zero element $(i,s,j)\in S_{i,j}$ of the semigroup $\textbf{BR}^{\mathbf{0}}(S,\theta)$ the family
  \begin{equation*}
    \mathscr{B}^{\mathbf{Ac}}_{\textbf{BR}}(i,s,j)=\left\{U_{i,j}\colon U\in\mathscr{B}_S(s) \right\}
  \end{equation*}
  is a base of the topology $\tau^{\mathbf{Ac}}_{\textbf{BR}}$ at the point $(i,s,j)\in \textbf{BR}^{\mathbf{0}}(S,\theta)$;

  \item[$(ii)$] the family $\mathscr{B}^{\mathbf{Ac}}_{\textbf{BR}}(\mathbf{0})=\left\{U_{(i_1,j_1),\ldots,(i_k,j_k)}\colon (i_1,j_1),\ldots,(i_k,j_k)\in\omega\times\omega\right\}$, where
      \begin{equation*}
        U_{(i_1,j_1),\ldots,(i_k,j_k)}=\textbf{BR}^{\mathbf{0}}(S,\theta)\setminus \left(S_{i_1,j_1}\cup\cdots\cup S_{i_k,j_k}\right),
      \end{equation*}
   is a base of the topology $\tau^{\mathbf{Ac}}_{\textbf{BR}}$ at zero $\mathbf{0}\in \textbf{BR}^{\mathbf{0}}(S,\theta)$.
\end{enumerate}
Obviously that $\tau^{\mathbf{Ac}}_{\textbf{BR}}$ is the topology of the Alexandroff one-point compactification of the Hausdorff locally compact space $\bigoplus\left\{S_{i,j}\colon i,j\in\omega\right\}$ with the remainder $\{\mathbf{0}\}$ (here for any $i,j\in\omega$ the space $S_{i,j}$ is homeomorphic to the compact semigroup  $(S, \tau_S)$ by the mapping $(i,s,j)\mapsto s$). Simple routine verifications show that the semigroup operation
in $\left(\textbf{BR}^{\mathbf{0}}(S,\theta),\tau^{\mathbf{Ac}}_{\textbf{BR}}\right)$ is separately continuous.
\end{example}

Lemmas~\ref{lemma-3.4} and \ref{lemma-3.8} imply the following dichotomy for locally compact Bruck--Reilly extensions of compact semitopological monoids in the class of Hausdorff semitopological semigroups with adjoined zero:

\begin{theorem}\label{theorem-3.11}
Let \emph{$\left(\textbf{BR}^{\mathbf{0}}(S,\theta),\tau_{\textbf{BR}}^{\mathbf{0}}\right)$} be a Hausdorff locally compact semitopological semigroup with a compact subsemigroup $S_{i,i}$ for some $i\in\omega$. Then \emph{$\left(\textbf{BR}^{\mathbf{0}}(S,\theta),\tau_{\textbf{BR}}^{\mathbf{0}}\right)$} is topologically isomorphic either to \emph{$\big(\textbf{BR}^{\mathbf{0}}(S,\theta),\tau^\oplus_{\textbf{BR}}\big)$} or to \emph{$\left(\textbf{BR}^{\mathbf{0}}(S,\theta),\tau^{\mathbf{Ac}}_{\textbf{BR}}\right)$}.
\end{theorem}

The following theorem describes the structure of inverse $0$-simple $\omega$-semigroups.

\begin{theorem}\label{theorem-3.12}
Every inverse $0$-simple $\omega$-semigroup is isomorphic to an inverse simple $\omega$-semigroup with adjoined zero.
\end{theorem}

\begin{proof}
Suppose $S$ is an inverse $0$-simple $\omega$-semigroup and $\mathbf{0}$ is zero of $S$. We shall show that $S\setminus\{\mathbf{0}\}$ is an inverse subsemigroup of $S$. Since $S$ is an inverse semigroup, we have that $x^{-1}\in S\setminus\{\mathbf{0}\}$ for a non-zero element $x$ from $S$.

Suppose to the contrary that there exist $x,y\in S\setminus\{\mathbf{0}\}$ such that $x\cdot y=\mathbf{0}$. If $x^{-1}$ and $y^{-1}$ are inverse elements of $x$ and $y$ in $S$, then $x^{-1}\neq \mathbf{0}\neq y^{-1}$. Then $x^{-1}\cdot x$ and $y\cdot y^{-1}$ are non-zero idempotents of $S$. Since $S$ is an inverse $0$-simple $\omega$-semigroup, we conclude that $(x^{-1}\cdot x)\cdot (y\cdot y^{-1})\neq \mathbf{0}$, but  $(x^{-1}\cdot x)\cdot (y\cdot y^{-1})=x^{-1}\cdot (x\cdot y)\cdot y^{-1}=x^{-1}\cdot \mathbf{0}\cdot y^{-1}=\mathbf{0}$, a contradiction. The obtained contradiction implies the statement of the theorem.
\end{proof}

The Kochin Theorem \cite{Kochin=1968} and Theorem~\ref{theorem-3.12} imply the following:

\begin{theorem}\label{theorem-3.13}
Every inverse $0$-simple $\omega$-semigroup $S$ is isomorphic to the Bruck–Reilly extension \emph{$\textbf{BR}^\mathbf{0}(T,\theta)$} of a finite chain of groups $T=\left[E;G_\alpha,\varphi_{\alpha,\beta}\right]$ with adjoined zero.
\end{theorem}

The main result of this section is the following theorem.

\begin{theorem}\label{theorem-3.14}
Let $S$ be a Hausdorff semitopological $0$-simple $\omega$-semigroup such that every maximal subgroup of $S$ is compact. Then $S$ is topologically isomorphic to the topological Bruck–Reilly extension \emph{$\left(\textbf{BR}^{\mathbf{0}}(T,\theta),\tau_{\textbf{BR}}^{\mathbf{0}}\right)$} of a finite semilattice $T=\left[E;G_\alpha,\varphi_{\alpha,\beta}\right]$ of compact groups $G_\alpha$ in the class of Hausdorff topological inverse semigroups with adjoined zero such that the topology \emph{$\tau_{\textbf{BR}}^{\mathbf{0}}$} induces on \emph{$\textbf{BR}^{\mathbf{0}}(T,\theta)$} the sum direct topology \emph{$\tau_{\textbf{BR}}^{\oplus}$}. Moreover, if the space of $S$ is locally compact, then either the space \emph{$\left(\textbf{BR}^{\mathbf{0}}(T,\theta),\tau_{\textbf{BR}}^{\mathbf{0}}\right)$} is compact or any $\mathscr{H}$-class in \emph{$\left(\textbf{BR}^{\mathbf{0}}(T,\theta),\tau_{\textbf{BR}}^{\mathbf{0}}\right)$} is open-and-compact.
\end{theorem}

\begin{proof}
The first statement of the theorem follows from Theorems~\ref{theorem-2.9} and \ref{theorem-3.13}. Next, using Theorem~\ref{theorem-3.11}, we obtain the second statement.
\end{proof}

\begin{remark}\label{remark-3.15}
We observe that the Bruck–Reilly extension $\textbf{BR}^{\mathbf{0}}(T,\theta)$ of a finite semilattice $T=\left[E;G_\alpha,\varphi_{\alpha,\beta}\right]$ of groups $G_\alpha$ with adjoined zero has two types of $\mathscr{H}$-classes: the first is a singleton and its consists of zero $\mathbf{0}$, and other classes  are of the form $(G_\alpha)_{i,j}$, $i, j\in\omega$.
\end{remark}

Since the bicyclic monoid $\mathscr{C}(p,q)$ does not embed into any Hausdorff compact topological semigroup \cite{Anderson-Hunter-Koch=1965}, Theorem~\ref{theorem-3.14} implies the following corollary.

\begin{corollary}\label{corollary-3.16}
Let $S$ be a Hausdorff topological $0$-simple inverse $\omega$-semigroup such that every maximal subgroup of $S$ is compact. Then $S$ is topologically isomorphic to the topological Bruck–Reilly extension \emph{$\left(\textbf{BR}^{\mathbf{0}}(T,\theta),\tau_{\textbf{BR}}^{\mathbf{0}}\right)$}
of a finite semilattice $T=\left[E;G_\alpha,\varphi_{\alpha,\beta}\right]$ of compact groups $G_\alpha$ in the class
of Hausdorff topological inverse semigroups with adjoined zero and any $\mathscr{H}$-class in \emph{$\left(\textbf{BR}^{\mathbf{0}}(T,\theta),\tau_{\textbf{BR}}^{\mathbf{0}}\right)$} is open-and-compact.
\end{corollary}

\section{On closures of simple inverse semitopological $\omega$-semigroup with compact maximal subgroups}\label{section-4}

Later we need the following lemma which is a simple generalization of Lemma~I.1$(i)$ from \cite{Eberhart-Selden=1969}.

\begin{lemma}\label{lemma-4.1}
Let  $\mathbf{BR}(T,\theta)$ be the Bruck-Reilly extension  of a monoid $T$. Then for arbitrary $T_{i_1,j_1}$ and $T_{i_2,j_2}$ of $\mathbf{BR}(T,\theta)$, $i_1,j_1,i_2,j_2\in\omega$, there exist finitely many subsets $T_{i,j}$ in $\mathbf{BR}(T,\theta)$,  $i,j\in\omega$, such that $T_{i_1,j_1}\cdot T_{i,j}\subseteq T_{i_2,j_2}$   $(T_{i,j}\cdot T_{i_1,j_1}\subseteq T_{i_2,j_2})$.
\end{lemma}

\begin{proof}
The definitions of the semigroup operations of the Bruck-Reilly extension $\mathbf{BR}(S,\theta)$ and the bicyclic monoid ${\mathscr{C}}(p,q)$ imply that if $(i_a,s_a,j_a)\cdot(i_x,s_x,j_x)=(i_b,s_b,j_b)$ in $\mathbf{BR}(S,\theta)$ then $(i_a,j_a)\cdot(i_x,j_x)=(i_b,j_b)$ in ${\mathscr{C}}(p,q)$. By Lemma~I.1$(i)$ of \cite{Eberhart-Selden=1969} every equation of the form $ax=b$ ($xa=b$) in ${\mathscr{C}}(p,q)$ has finitely many solutions, which implies the statement of the lemma.
\end{proof}

The following proposition generalizes Theorem I.3 from \cite{Eberhart-Selden=1969} and corresponding proposition from \cite{Gutik=2015}.

\begin{proposition}\label{proposition-4.2}
Let $T$ be a compact Hausdorff topological semigroup and \emph{$\left(\textbf{BR}(T,\theta),\tau_{\textbf{BR}}\right)$}  be a topological Bruck–Reilly extension of $T$ in the class of Hausdorff semitopological semigroups. If \emph{$\left(\textbf{BR}(T,\theta),\tau_{\textbf{BR}}\right)$} is a dense subsemigroup of a Hausdorff semitopological monoid $S$ and $I=S\setminus \textbf{BR}(T,\theta)\neq\varnothing$, then $I$ is a two-sided ideal of the semigroup $S$.
\end{proposition}

\begin{proof}
Fix an arbitrary element $y\in I$. If $(i,s,j)\cdot y=z\notin I$ for some $(i,s,j)\in \textbf{BR}(T,\theta)$, then $z=(k,t,l)\in \textbf{BR}(T,\theta)$ for some $t\in T$ and $k,l\in\omega$. By Theorem~\ref{theorem-2.9} there exists an open neighbourhood $U(y)$ of the point $y$ in the space $S$ such that $\{(i,s,j)\}\cdot U(y)\subseteq T_{k,l}$. Since $T$ be a compact Hausdorff topological semigroup, Theorem~\ref{theorem-2.9} implies that the topology $\tau_{\textbf{BR}}$ coincides with the sum direct topology $\tau_{\textbf{BR}}^\oplus$. By Proposition 2.4 of \cite{Gutik=2018} all subsets of the form $T_{n,m}$, $n,m\in\omega$, are compact. Hence the neighbourhood $U(y)$ intersects infinitely many sets of the form $T_{n,m}$, $n,m\in\omega$. Then the semigroup operation of $\textbf{BR}(T,\theta)$ implies that $\{(i, s, j)\}\cdot U(y)\nsubseteq T_{k,l}$, which contradicts Lemma~\ref{lemma-4.1}. The obtained contradiction implies that $(i,s,j)\cdot y\in I$. The proof of the statement that $y\cdot(i,s,j)\in I$ for all $(i,s,j)\in \textbf{BR}(T,\theta)$ and $y\in I$ is similar.

Suppose to the contrary that $xy=w=(k,t,l)\in \textbf{BR}(T,\theta)$ for some $x,y\in I$. Theorem~\ref{theorem-2.9} and the separate continuity of the semigroup operation in $S$ implies that there exist open neighbourhoods $U(x)$ and $U(y)$ of the points $x$ and $y$ in $S$, respectively, such that $\{x\}\cdot U(y)\subseteq T_{k,l}$ and $U(x)\cdot \{y\}\subseteq T_{k,l}$. By Proposition 2.4 of \cite{Gutik=2018} all subsets of the form $T_{n,m}$, $n,m\in\omega$, are compact. Hence the neighbourhood $U(y)$ intersects infinitely many sets of the form $T_{n,m}$, $n,m\in\omega$, and hence both inclusions  $\{x\}\cdot U(y)\subseteq T_{k,l}$ and $U(x)\cdot \{y\}\subseteq T_{k,l}$ contradict mentioned above Lemma~\ref{lemma-4.1}. The obtained contradiction implies that $xy\in I$.
\end{proof}

Later we need the following trivial lemma, which follows from the separate continuity of the semigroup operation in semitopological semigroups.

\begin{lemma}\label{lemma-4.3}
Let $S$ be a Hausdorff semitopological semigroup and $I$ be a compact ideal in $S$. Then the Rees-quotient semigroup $S/I$ with the quotient topology is a Hausdorff semitopological semigroup.
\end{lemma}

\begin{theorem}\label{theorem-4.4}
Let $T$ be a compact Hausdorff semitopological semigroup and \emph{$\left(\textbf{BR}(T,\theta),\tau_{\textbf{BR}}\right)$}  be a topological Bruck–Reilly extension of~$T$ in the class of Hausdorff semitopological semigroups. Let \emph{$\textbf{BR}_I(T,\theta)=\textbf{BR}(T,\theta)\sqcup I$} and $\tau$ be a Hausdorff locally compact shift-continuous topology on \emph{$\textbf{BR}_I(T,\theta)$}, where $I$ is a compact ideal of $\textbf{BR}_I(T,\theta)$. Then either \emph{$(\textbf{BR}_I(T,\theta), \tau)$} is a compact semitopological semigroup or the ideal $I$ is open.
\end{theorem}

\begin{proof}
Suppose that the ideal $I$ is not open. By Lemma~\ref{lemma-4.3} the Rees-quotient semigroup $\textbf{BR}_I(T,\theta)/I$ with the quotient topology $\tau_\mathbf{q}$ is a semitopological semigroup. Let $\pi\colon \textbf{BR}_I(T,\theta)\to \textbf{BR}_I(T,\theta)/I$ be the natural homomorphism which is a quotient map. It is obvious that the Rees-quotient semigroup $\textbf{BR}_I(T,\theta)/I$ is isomorphic to the Bruck–Reilly extension with adjoined zero $\textbf{BR}^{\mathbf{0}}(T,\theta)$  and the image $(I)\pi$ is zero $\mathbf{0}$ of the semigroup $\textbf{BR}^{\mathbf{0}}(T,\theta)$.

We show that the natural homomorphism $\pi\colon \textbf{BR}_I(T,\theta)\to \textbf{BR}_I(T,\theta)/I$ is a hereditarily quotient map. In particular, we show that for every open neighbourhood $U(I)$ of the compact ideal $I$ in $(\textbf{BR}_I(T,\theta)/I,\tau_\mathbf{q})$  the image $(U(I))\pi$ is an open neighbourhood of the zero $\mathbf{0}$ in $(\textbf{BR}_I(T,\theta)/I,\tau_\mathbf{q})$. Indeed, $\textbf{BR}_I(T,\theta)/I\setminus U(I)$ is a closed
subset of $(\textbf{BR}_I(T,\theta)/I,\tau_\mathbf{q})$. Also, since the restriction \linebreak $\pi{\upharpoonleft}_{\textbf{BR}(T,\theta)}\colon \textbf{BR}(T,\theta)\to (\textbf{BR}(T,\theta))\pi$ of the natural homomorphism
$\pi\colon \textbf{BR}_I(T,\theta)\to \textbf{BR}_I(T,\theta)/I$ is one-to-one, $(\textbf{BR}_I(T,\theta)/I\setminus U(I))\pi$ is a closed subset of
$(\textbf{BR}_I(T,\theta)/I,\tau_\mathbf{q})$. Hence, $(U(I))\pi$ is an open neighbourhood of the zero $\mathbf{0}$ of the semigroup
$(\textbf{BR}_I(T,\theta)/I,\tau_\mathbf{q})$, and this implies that  the natural homomorphism $\pi\colon \textbf{BR}_I(T,\theta)\to \textbf{BR}_I(T,\theta)/I$ is a hereditarily quotient map.

Since $I$ is a compact ideal of the semitopological semigroup $(\textbf{BR}_I(T,\theta),\tau)$, the preimage $(y)\pi^{-1}$ is a compact subset of $(\textbf{BR}_I(T,\theta),\tau)$ for every $y\in \textbf{BR}_I(T,\theta)/I$. By the Din’ N’e T’ong Theorem the image of a locally compact Hausdorff space under a hereditary quotient map with compact fibers into a Hausdorff space is locally compact (see \cite{Din-Ne-Tong=1963} or \cite[3.7.E]{Engelking=1989}), and hence the space  $(\textbf{BR}_I(T,\theta)/I,\tau_\mathbf{q})$ is Hausdorff and locally compact. Since the ideal $I$ is not open, by Theorem~\ref{theorem-3.14} the semitopological semigroup $(\textbf{BR}_I(T,\theta)/I,\tau_\mathbf{q})$ is topologically isomorphic to $\left(\textbf{BR}^{\mathbf{0}}(T,\theta),\tau^{\mathbf{Ac}}_{\textbf{BR}}\right)$, and hence, it is compact.

Next, we show that the space $(\textbf{BR}_I(T,\theta),\tau)$ is compact. Let $\mathscr{U}=\{U_\alpha\colon \alpha\in \mathscr{I}\}$ be any open cover of $(\textbf{BR}_I(T,\theta),\tau)$. Since the ideal $I$ is compact, it can be covered by some finite subfamily $\mathscr{U}'=\{U_{\alpha_1}, \ldots, U_{\alpha_k}\}$ of $\mathscr{U}$. Put $U=U_{\alpha_1}\cup \cdots\cup U_{\alpha_k}$. Then $\textbf{BR}_I(T,\theta)\setminus U$ is a closed subset of $(\textbf{BR}_I(T,\theta),\tau)$. Since the restriction $\pi{\upharpoonleft}_{\textbf{BR}(T,\theta)}\colon \textbf{BR}(T,\theta)\to (\textbf{BR}(T,\theta))\pi$ of the natural homomorphism $\pi\colon \textbf{BR}_I(T,\theta)\to \textbf{BR}_I(T,\theta)/I$ is one-to-one, the image $(\textbf{BR}_I(T,\theta)\setminus U)\pi$
is a closed subset of the space $(\textbf{BR}_I(T,\theta)/I,\tau_\mathbf{q})$, and hence the image $(\textbf{BR}_I(T,\theta)\setminus U)\pi$ is compact, because the semitopological semigroup $(\textbf{BR}_I(T,\theta)/I,\tau_\mathbf{q})$ is compact. Thus, the set $\textbf{BR}_I(T,\theta)\setminus U$ is compact, and hence,  there exists a finite subfamily $\mathscr{U}''$ of $\mathscr{U}$, which is an open cover of $\textbf{BR}_I(T,\theta)\setminus U$. Then $\mathscr{U}'\cup\mathscr{U}''$ is a finite cover of the space $(\textbf{BR}_I(T,\theta),\tau)$$(\textbf{BR}_I(T,\theta),\tau)$. Hence the space $(\textbf{BR}_I(T,\theta),\tau)$ is compact, too.
\end{proof}

Theorem~\ref{theorem-4.4} implies the following:

\begin{theorem}\label{theorem-4.5}
Let $S$ be a Hausdorff semitopological simple inverse $\omega$-semigroup such that every maximal subgroup of $S$ is compact. Let $S_I=S\sqcup I$, $\tau$ be a Hausdorff locally compact shift-continuous topology on $S_I$, and $I$ be a compact ideal of $S_I$. Then either $(S_I, \tau)$ is a compact semitopological semigroup or the ideal $I$ is open.
\end{theorem}

Since every Bruck–Reilly extension of a monoid contains an isomorphic copy the bicyclic monoid $\mathscr{C}(p, q)$ and compact topological semigroups
do not contain the semigroup $\mathscr{C}(p, q)$, Theorem~\ref{theorem-4.4} implies the following corollary.

\begin{corollary}\label{corollary-4.6}
Let $T$ be a compact Hausdorff topological semigroup and \emph{$\left(\textbf{BR}(T,\theta),\tau_{\textbf{BR}}\right)$}  be a topological Bruck–Reilly extension of $T$ in the class of Hausdorff semitopological semigroups. Let \emph{$\textbf{BR}_I(T,\theta)=\textbf{BR}(T,\theta)\sqcup I$} and $\tau$ be a Hausdorff locally compact shift-continuous topology on \emph{$\textbf{BR}_I(T,\theta)$}, where $I$ is a compact ideal of $(\textbf{BR}_I(T,\theta),\tau)$. Then  the ideal $I$ is open in $(\textbf{BR}_I(T,\theta),\tau)$.
\end{corollary}

Corollary~\ref{corollary-4.6} implies

\begin{corollary}\label{corollary-4.7}
Let $S$ be a Hausdorff semitopological simple inverse $\omega$-semigroup such that every maximal subgroup of $S$ is compact. Let $S_I=S\sqcup I$, $\tau$ be a Hausdorff locally compact semigroup topology on $S_I$, and $I$ be a compact ideal of $S_I$. Then the ideal $I$ is open in $(S_I, \tau)$.
\end{corollary}



\end{document}